\documentclass[12pt]{article}
\usepackage[english]{babel}
\usepackage{amsthm,amsfonts, amsbsy, amssymb,amsmath,graphicx}
\usepackage{graphics}
\usepackage{xcolor}

\newtheorem{thm}{Theorem}

\theoremstyle{definition}
\newtheorem{definition}{Definition}

\theoremstyle{remark}
\newtheorem{remark}{Remark}

\newcommand{\Z}{\mathbb Z}

\title{Cobordism invariants for knots with two indices}

\author{V.O.Manturov \footnote{Moscow Institute of Physics and Technology, Kazan Federal University,
}}

\begin{document}
\maketitle

\begin{abstract}
We construct an invariant of virtual knots which is a sliceness obstruction
and sensitive to the $\Delta$-move.

This invariants works if $\Z_{2}\oplus \Z_{2}$-index of chords is present.
\end{abstract}



\section{Preliminaries}
We work with any knot theory where two independent parities are present\footnote{This may happen, for example, when
we consider knots in the thickened torus}. We are interested in crossings which are
odd in at least one of those parities. There are three options for such crossings, which we call {\em crossings of types $a,b,c$;
these are three non-trivial elements of the group $\Z_{2}\oplus \Z_{2}$.}
For types $a,b,c$ the most important property that we
shall use will be: if crossings of two different indices (say, $a,b$) take part in a Reidemeister-3 move then the third crossing
has the third possible index $(c)$.

Those chords (and crossings) which are trivial with respect to both parities will be called {\em chords (crossings) of trivial type (index).} 

The invariants will first be constructed for knots with a reference point (long knots) but then we'll see
what happens if we move a reference point through a crossing.
The invariant for compact knots remains meaningful if some parity condition for chords of types $a$,$b$,$c$ is satisfied.

\subsection{Why is this construction important}

For the present time, there are lots of constructions of virtual knot invariant which work as follows:
we charge each crossing with some information (writhe number, parity, index, homotopical index, etc.) 
and then take a certain sum over all (non-trivial) crossings of sum function depending on these types, indices etc
(see \cite{Cheng,IMN}).
These invariants by definition turn out to be invariant under the $\Delta$-move.

Here we are constructing an invariant of ``group-type'' which is manifestly sensitive to the $\Delta$-move.

\subsection{Acknowledgements}

I am extremely grateful to Seongjeong Kim for various discussion during the preparation of this text.

The work of V.O.Manturov was funded by the development program of the Regional
Scientific and Educational Mathematical Center of the Volga Federal District, agreement
N 075-02-2020.

\section{The construction}

From now on we work only for a knot theory with two parities and over/under information at crossings.

Consider a Gauss diagram $D$ of a long knot. We are interested in chords of indices $a,b,c$.

         Now, we shall associate letters $a,a',A,A'$ to chord ends for chords of index $a$.
Similarly, we associate $b,b',B,B'$ chords of index $b$. We don't associate any letters with primes to chords of index $C$ at all
[however we keep in mind that they exist: small ones ($c$) and capital ones ($C$)].

         To be more precise, we use small letters (with or without primes) for undercrossing ends,
and we use capital letters for overcrossing ends.

          The rule for the prime is as follows: for a chord end of type $a$ (potentially it can become $a$ or $a'$) we
look at the opposite chord end ($A$ or $A'$) and look how many
chord ends $c$ precede this $A$ (or $A'$). Depending on the parity, we justify
$a$ (or $a'$). We do exactly the same for $b,b'$: look at the opposite end ($B$ or $B'$)
and count the parity of the number of preceding $c,c'$.

More precisely, we put prime in those cases when the respective parity is odd.

As for the capital letters $A,A',B',B'$ we just swap capital and small letters:
e.g., for a chord end of type $A$ (it may become $A$ or $A'$)
we look at the opposite chord ($a$ or $a$) and look how many chord ends $C$ precede this $a$
(or $a'$).

We denote the obtained word by $w(D)$.

The interested reader can ask {\em why do we conut the number of chord ends before the given
chord end, not after?} The answer is: if we replace ``before'' with ``after'', this will work nicely 
as well. Hence, the invariant $w$ to be constructed will have its ``twin'' $w_{after}$.
To see the knots where these two definitions are equal, we introduce the following
\begin{definition}
A knot is {\em $c$-even} (resp., {\em $a$-even, $b$-even}) if the number of chords of index $c$ of some (hence, any)
of its representatives is even (resp., the number of $a$-chords, $b$-chords).
Otherwise, we say that knot diagram (hence, knot) is $c$-odd (resp., $a$-odd, $b$-odd)).
\end{definition}

\begin{remark}
Oddness in any of these senses yields non-triviality and non-sliceness in various senses. 
\end{remark}

Our next goal is to understand the relations in the group are forced by Reidemeister moves.

\subsection{Reidemeister moves. Group relations}
            Now we want to apply Reidemeister moves and see which relations we get for the
letters.

{\bf Reidemeister 1 move.}

Nothing happens at all.
{\bf Reidemeister 2 move.} Two neigbouring chord ends have the same index, they are either both
over or both under, and they either both have prime or both do not have prime. Hence, we have two cancellations of types $aa=1, ..., B'B'=1$.

{\bf Reidemeister 3.} 

We have one of the three options:

$I$ All chords have trivial index. Nothing happens to any of our words.
 
$II.$ 
There are two chords of the same index  taking part in the move.
Then the third chord has trivial index.
This leads to a commutativity of letters of the same index:

$$[a,A]=[a,A']=[a',A]=[a',A']=[b,B]=[b,B']=[b',B]=[b',B']=1 $$

and

$$[a,a']=[b,b']=[A,A']=[B,B']=1.$$

$III.$ Finally, assume that all three chords $a,b,c$ are present.
Now, we make an important observation that
{\bf among the three pieces (ab),(ac),(bc) there is one piece with two consequent overcrossings,
one piece with two consequent undercrossings and one piece with an overcrossing and an undercrossing}.

 We are interested in three pieces of chord diagrams:
the (ab)-piece, the (ac)-piece and the (bc)-piece.

Note that in the (ac) piece the letter of type a does not change at all.
Indeed: looking at the chord end of index $a$ (it may be one of $a,a', A,A'$),
the opposite chord end (say, $A$) is adjacent to some chord end of index $b$ (not $c$!)
So, after the Reidemeister move the count of $c$ does not change,
which means that the letter of type $a$
in the (ac)-piece stays the same.
As well as the letter b in the (bc)-piece.

So, the only interesting thing happens in the subword
(ab) (which can be,say, $aB'$ or $BA'$: 32 variants in total).

Let us see in more details what happens to such a word if we apply the third Reidemeister moves.

Here we have two subcases

$III.a$ First, if one of the letters are capital and the other is small then this leads to a commutativity:
$[a,B]=[a,B']=[b,A]=[b,A']=[a',B]=[a',B']=[b',A]=[b',A']=1.$

Indeed, if one of the letter is capital and the other is small (say, $A$ and $b$), then in the other two
pieces we have

($a$ or $a'$) near $c$
and
($B$ or $B'$) near $C.$

Indeed, for a true third Reidemeister move there should be {\bf one over-piece and one under-piece}\footnote{
unlike the $\Delta$-move where the
situation with three pieces like
$aB, bC, cA$ is possible.}.

This is the third Reidemeister move not $\Delta$ move: one of the three pieces has to be twice ``over'', the other
has to be twice ``under''.

This means that none of the two letters (in our case $A,b$) acquires (looses) the prime.
Which means that we have just commutativity.

$III.b$ The most interesting case is when both letters of types $a,b$ on the $a,b$ piece are small
or both are capital.
In this case exactly one of them acquires (looses) its prime, which means that we get the following  relations:

$$
ab=b'a=ba'=a'b',
ba=a'b=ab'=b'a',
$$

$$
AB=B'A=BA'=A'B',
BA=AB'=AB'=B'A'.
$$

Denote the group with the above generators and relations by $D$.

We formalise the above consideration to the following
\begin{thm}
$w(K)\in D$ is an invariant of long knots valued in $D$, more precisely, if $K,K'$ are equivalent
long knots (in the theory with two parities) then $w(K)=w(K')$ as elements of $D$.
\label{th1}
\end{thm}

\begin{remark}
The same is true for the invariant $w_{after}$.
\end{remark}

\begin{remark}
It is important to note that for a chord of type b the presence of prime on
one end doesn't mean that the prime exists on the other end. Namely, $b$ might be coupled with $B'$,
$b'$ might be coupled with $B$.
\end{remark}




\subsection{Simplifying the group}
How to handle the above group.
We denote the above group by

\begin{eqnarray*}
D=\langle a,b,a',b',A,B,A',B'&|& aa=1, ..., B'B'=1,\cr
&&[a,a']=[a,A']=... = [A,A']=1,\cr
&&[b,b']=[b,B']=... = [B,B']=1,\cr
&&[a,B]=[a,B']=[a',B]=[a',B']= \\
&&=[b,A]=[b,A']=[b',A]=[b',A']=1; \cr
&&ab=b'a=ba'=a'b',\cr
&&ba=ab'=a'b=b'a',\cr
&&AB=B'A=BA'=A'B',\cr
&&BA=AB'=A'B=B'A'
\rangle.\end{eqnarray*}

To simplify the presentation of the group $D$, we just express all letters with
primes in terms of the letters without primes just by conjugation, say, $a'=bab$.

Substituting $a'=bab$ to $[a,a']=1$, we get 
$$(ab)^{4}=1;$$ we get the same for
substituting $b'=aba$ to $[b,b']=1$ and similarly, we get
$$(AB)^{4}=1$$ for $[A,A']=[B,B']=1$.

One can easily check that the above two relations are sufficient.
For example, $ab=a'b'$ yields $ab=bababa$ which is equivalent to $(ab)^{4}=1,$
or, say,
$b'ab' = a'$ means that
$aba a aba = bab$ which yields $ababa = bab \Longleftrightarrow (ab)^{4} =1.$

The same happens with $A$ and $B$.

Hence, our group $D$ is isomorphic to the direct sum of two Coxeter groups:

$$ D = \langle a,b|a^{2}=b^{2}=(ab)^{4}=1\rangle \oplus \langle A,B|A^{2}=B^{2}= (AB)^{4}=1\rangle.$$

%



\subsection{What to do with compact knots}

Here we restrict ourselves only to {\em knots with even number of chords of index $c$}.

We recall that Vassiliev invariants have Viro-Polyak combinatorial formulae (Goussarov's theorem, \cite{GPV})
which count some combinations of subdiagrams in the Gauss diagram.

The simplest example is the combinatorial formula for the second Conway's coefficient $c_{2}$ 
counts the subdiagrams  for the {\em long knot} as shown in Fig. \ref{Fig1}:

\begin{figure}

\centering\includegraphics[width=110pt]{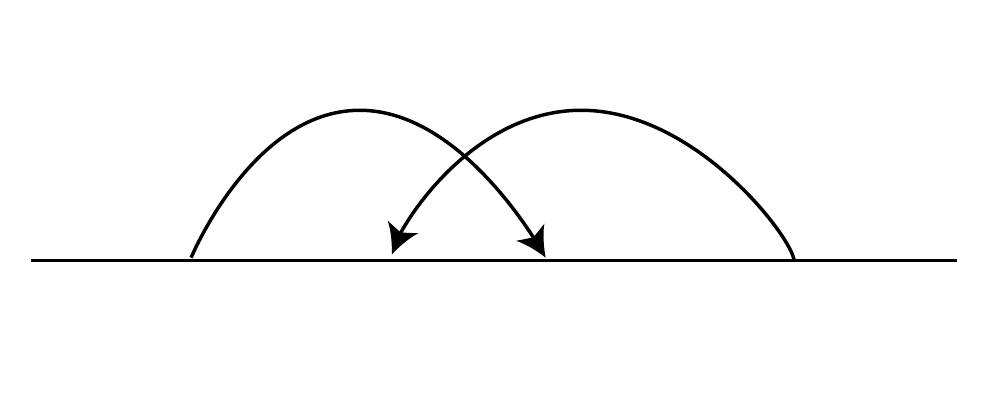}
\caption{The diagram from the combinatorial formula for $c_{2}$}
\label{Fig1}
\end{figure}

Interestingly, this formula uses long knots (not compact knots) and the proof that it gives an invariant
of compact classical knots uses some non-local argument: we should the leftmost chord end and take it
to the rightmost position.

Note also that this trick works only with {\em classical} knots and {\em some classes of virtual knots};
generally, virtual knots are not the same as long virtual knots, see \cite{Long}.

In a similar way, we have to restrict the class of {\em compact knots} in consideration to {\em $c$-even knots}.

Namely, in the second part of this section we deal only with $c$-even knots.

\begin{thm}
Let $K,K'$ be two isotopic oriented $c$-even virtual knots. Then $w(K)$ and $w(K')$ can be obtained from
each other by 
conjugations and automorphisms of $D$ of the following forms:

$$\phi: D\to D, \phi(a)=bab,\phi(b)=aba, \phi(A)=A,\phi(B)=B;$$

$$\psi: D\to D, \phi(a)=a,\phi(b)=b, \phi(A)=BAB,\phi(B)=ABA;$$

$$\chi: D\to D, \chi=\phi\circ\psi=\psi\circ \phi.$$
\end{thm}

\begin{proof}
From Theorem \ref{th1} we only have to check what happens to the word $w(K)$ if we change the
reference point, namely, if the move the leftmost point to the right.

If this chord has trivial index then nothing happens to prime arrangement nor to the word.

If this chord has index $c$ then the evenness of the knot comes into play:
for each chord of index $a$, for its chord ends $a$ and $A$ it does not matter
whether we count chords of index $c$ before it or after it. 
So, moving $c$ from one end to the other end) swaps primes for all
$a,a',b,b'$ and preserves all primes for $A,A',B,B'$. Hence our word undergoes the map $\phi$.

If we move $C$ then we the word undergoes the map $\psi$.

Finally, if we move some chord of index $a$ or $b$ primes stay the
same as before, but the word undergoes conjugation.

\end{proof}

\section{Post Scriptum. On cobordisms}

Actually, all indices, capital letters, and primes are defined in a way which may be
generalised from knots to $2$-knots with boundary. So, potentially this will work for 
cobordisms between two knots provided that these cobordisms respect the two parities.

We shall formulate our next goal in the form of 

{\bf Conjecture.}
$w(D)$ is a sliceness obstruction for $2$-knots with two parities.

\end{document}